\documentclass{article}
\usepackage{amsmath}
\usepackage{amssymb}
\usepackage{amsfonts}
\usepackage{mitpress}
\usepackage{mathtools}

\newtheorem{theorem}{Theorem}

\newtheorem{corollary}[theorem]{Corollary}

\newtheorem{definition}[theorem]{Definition}
\newtheorem{Defns}[theorem]{Definitions}
\newtheorem{example}[theorem]{Example}
\newtheorem{examples}[theorem]{Examples}

\newtheorem{lemma}[theorem]{Lemma}
\newtheorem{notation}[theorem]{Notation}

\newtheorem{proposition}[theorem]{Proposition}
\newtheorem{remark}[theorem]{Remark}

\usepackage[hidelinks]{hyperref}
\usepackage{subcaption}
\usepackage{subfiles}
\usepackage{float}
\usepackage{graphicx}

\newenvironment{proof}[1][Proof]{\noindent\textbf{#1.} }{\ \rule{0.5em}{0.5em}}
\newdimen\dummy
\dummy=\oddsidemargin
\begin{document}

\title{Types of Irreducible Divisor Graphs of Noncommutative
Domains, II}
\author{A. Naser$^{1}$, R. E. Abdel-Khalek$^{2}$, R. M. Salem$^{2}$, and A.
M. Hassanein$^{2}$ \\
$^{1}$Department of Mathematics, Faculty of Women for Arts, \\
Science and Education, Ain Shams University, Cairo, Egypt\\
$^{2}$Department of Mathematics, Faculty of Science,\\
Al-Azhar University, Nasr City 11884, Cairo, Egypt\\
asmaanaser99@yahoo.com, ramy\_ama@yahoo.com,\\
rsalem\_02@hotmail.com, mhassanein\_05@yahoo.com\\
\textbf{Dedicated to the memory of our professor M. H. Fahmy}}
\maketitle

\begin{abstract}
In this paper, we continue investigation of the directed and undirected irreducible divisor graph concepts $G(x)$ and $\Gamma (x)$ of $x\in D^{\ast} \backslash U(D)$, respectively, which were introduced in \cite{paperphd1}. Consequently, we introduce two generalizations of these concepts. The first one is the irreducible divisor simplicial complex $S(x)$ of $x\in D^{\ast} \backslash U(D)$ in a noncommutative atomic domain $D$, which simultaneously extends the commutative case that was introduced by R. Baeth and J. Hobson in \cite{simplicial}. The second one is the directed and undirected $\tau $-irreducible divisor graphs $G_{\tau }(x)$ and $\Gamma _{\tau }(x)$ of $x\in D^{\ast} \backslash U(D)$, respectively, in a noncommutative $\tau $-atomic domain $D$ with a symmetric and associate preserving relation $\tau $ on $D^{\ast} \backslash U(D)$. Those graphs also extend the commutative case that was introduced by C. P. Mooney in \cite{generalirreducible}. Furthermore, we extend the results of \cite{simplicial} and \cite{generalirreducible} to give a characterization of n-unique factorization domains via those two generalizations.
\end{abstract}
\section{Introduction}
Throughout this paper,
all domains are associative and noncommutative with identity unless otherwise stated. We denote the nonzero elements of a domain $D$, the group of units in $D$, and the nonzero nonunit elements of $D$ by $D^{\ast }$, $U(D)$, and $D^{\#}$, respectively. The concept of an irreducible divisor graph of $x\in D^{\#}$ in a commutative atomic domain $D$, denoted by $G(x)$, whose vertices are the nonassociate irreducible divisors of $x$ and two vertices $y$ and $z$ have an edge between them if and only if $yz$ divides $x$, see \cite{irreducible}. The author in \cite{irreducible} investigated the interplay between the ring theoretic properties of a domain $D$ and the graph theoretic properties of $G(x)$ for all $x$ in $D^{\#}$. More precisely, they showed that if $D$ is a commutative atomic domain, then $D$ is a UFD if and only if $G(x)$ is complete for every $x\in D^{\#}$ if and only if $G(x)$ is connected for every $x\in D^{\#}$ \cite[Theorem 5.1]{irreducible}. Then, M. Axtell et al. in \cite{comprees} introduced an alternate irreducible divisor graph called the compressed irreducible divisor graph of $x\in D^{\#}$ in a commutative atomic domain $D$, denoted by $G_{c}(x)$. Also, they showed that if $D$ is a commutative atomic domain, then $D$ is a UFD if and only if $G_{c}(x)$ $\cong K_{1}$, the complete graph of one vertex, for all $x\in D^{\#}$ \cite[Theorem 5.7]{comprees}.\\
After that, the irreducible divisor graph has been studied and developed by many authors, see \cite{simplicial} and \cite{generalirreducible}. In \cite{simplicial}, Baeth and Hobson introduced the irreducible divisor simplicial complex of $x\in D^{\#}$ in a commutative atomic domain $D$, whose vertices are the nonassociate irreducible $y$ divisors of $x$ and any set of vertices $\{y_{1},y_{2},...,y_{n}\}$ forms a face if and only if $y_{1}y_{2}...y_{n}$ divides $x$, denoted by $S(x)$. In light of this concept, they identified the corresponding characterizations of a UFD. This concept gives a higher-dimensional notion of the irreducible divisor graph $G(x)$. On the other hand, C. P. Mooney in \cite{generalirreducible} introduced the concept of a $\tau $-irreducible divisor graph of $x\in D^{\#}$ in a commutative $\tau $-atomic domain $D$ with a symmetric and associate preserving relation $\tau $ on $D^{\#}$. The $\tau $-irreducible divisor graph of $x\in D^{\#}$ is the graph $G_{\tau }(x)$, whose vertices are the nonassociated $\tau $-irreducible divisors of $x$ and two vertices $y$ and $z$ have an edge between them if and only if $yz$ $\tau $-divides $x$. They showed that if $D$ is a commutative $\tau $-atomic domain, then $D$ is a $\tau $-UFD if and only if $G_{\tau }(x)$ is complete for all $x\in D^{\#}$ if and only if $G_{\tau }(x)$ is connected for all $x\in D^{\#}$ \cite[Theorem 4.3]{generalirreducible}.\\
In \cite{paperphd1}, the authors extended the concepts and results of the irreducible divisor graph and the compressed irreducible divisor graph to the noncommutative setting. In this paper, we generalize the concepts and results presented in \cite{simplicial} and \cite{generalirreducible} to the noncommutative setting. In Section 2, we present the basic definitions and notations that are exploited in our study. Section 3 gives the concept of the irreducible divisor simplicial complex $S(x)$ of $x\in D^{\#}$ in a noncommutative atomic domain $D$. It is a higher-dimensional version of the notion of the noncommutative irreducible divisor graph $G(x)$ appeared in \cite{paperphd1}. Moreover, we give some examples of the noncommutative irreducible divisor simplicial complex of $x\in D^{\#}$ in a noncommutative atomic domain $D$. We also generalize the main results in \cite{simplicial} to the case of a noncommutative domain. In Section 4, we introduce the directed and undirected $\tau $-irreducible divisor graphs $\Gamma _{\tau}(x) $ and $G_{\tau }(x)$ of $x\in D^{\#}$, respectively, in a noncommutative $\tau $-atomic domain $D$ with a symmetric and associate preserving relation $\tau $ on $D^{\#}$. Also, we give some examples of those graphs. Moreover, we show that if $D$ is a $\tau $-n-atomic domain, then $D$ is a $\tau $-n-FFD if and only if $\Gamma _{\tau }(x)$ is finite for all $x\in D^{\#}$ if and only if for all $x\in D^{\#}$, outdeg$(w)$ and indeg$(w)$ are finite for all $w\in V(\Gamma _{\tau }(x))$  if and only if for all $x\in D^{\#}$, outdegl$(w)$ and indegl$(w)$ are finite for all $w\in V(\Gamma _{\tau }(x))$. Also, we obtain the corresponding result to Theorem 4.3 in \cite{generalirreducible} in a noncommutative domain. More precisely, we study under what conditions the following are equivalent:
\begin{enumerate}
 \item $D$ is a $\tau $-n-UFD;
\item $\Gamma _{\tau }(x)$ is a tournament for all $x\in D^{\#}$;
\item $\Gamma _{\tau }(x)$ is unilaterally connected for all $x\in D^{\#}$;
\item $\Gamma _{\tau }(x)$ is weakly connected for all $x\in D^{\#}$.
\end{enumerate}
We also give the corresponding result in the undirected case.

\section{Definitions and Notations}
This section contains the basic definitions relating to algebra and graphs, which will be used throughout the paper. Unless otherwise stated, all domains are noncommutative and associative with identity.

\begin{Defns}
\begin{description}
\item[i.] An element $a\in D$ is called irreducible (atom) if $a$ is a nonzero nonunit that is not the product of two nonunits. $Irr(D)$ denotes the set of all irreducible elements in a domain $D$ and $\overline{Irr}(D)$ indicates a (pre-chosen) set of coset representatives, one representative from each coset in the collection $\{aU(D):a\in Irr(D)\}$.
\item[ii.] An element $a\in D$\ is called a right (resp. left) divisor of $b\in D$, denoted by $a\mid _{r}b$ (resp. $a\mid _{l}b$), if $b=ac$ (resp. $b=ca$), and the element $a$\ is a divisor of $b$,\ denoted by $a\mid b$, if $b=cad$ where $c,d\in D$.
\item[iii.] An element $a\in D$\ is called right (resp. left) associated with $b\in D$, denoted by $a\sim _{r}b$ (resp. $a\sim _{l}b$), if $a=bu$ (resp. $a=ub$) for some unit $u$ $\in U(D),$ and the elements $a,b$ are associates,\ denoted by $a\sim b$, if there are $u,v\in U(D)$ such that $a=vbu$.
\item[iv.] An element $a\in D$\ is called normal in a domain $D$ if $aD=Da$, and $D$ is normal if all its elements are normal.
\item[v.] A normal element $p\in D$\ is called prime if $p\mid ab$ implies $p\mid a$ or\ $p\mid b$, where $a,b\in D$.
\end{description}
\end{Defns}
\begin{definition} \cite{natomic}
A domain $D$ is called a unique factorization domain, for short UFD (resp. normal unique factorization domain, for short n-UFD), if $:$\\
(1) the domain $D$ is atomic (resp. normal atomic, for short n-atomic), i.e. for every $r\in D^{\#}$ there exist irreducible (resp. normal irreducible) elements $r_{1},r_{2},\dots ,r_{l}$ such that $r=r_{1}r_{2}\dots r_{l}$, and\\
(2) if $r\in D^{\#}$ has two atomic (resp. n-atomic) factorizations, i.e. $r=r_{1}r_{2}\dots .r_{l}=t_{1}t_{2}\dots .t_{m}$ \ where $r_{1},r_{2},\dots,r_{l},t_{1},t_{2},\dots ,t_{m}$ are irreducible (resp. normal irreducible) elements, then $l=m$ and there is a permutation $\sigma \in S_{l}$ such that for each $i=1,\dots ,l,$ $r_{i}$ is associated with $t_{\sigma (i)}$.
\end{definition}
\begin{notation}
Let $(V,E)$ be a graph with the set of vertices $V$ and the set of edges $E$. We denote an edge in the digraph $\Gamma $ from vertex $a$ to vertex $b$ as $(a,b)$, noting that the edge $(a,b)$ is not the same as the edge $(b,a)$. While the edge in the undirected graph $G$ between vertices $a$ and $b$ as $\{a,b\}$, note that the edge $\{a,b\}$ is the same as the edge $\{b,a\}$.
\end{notation}
\begin{Defns}
\begin{description}
\item[i.] The digraph $\Gamma $ is called complete (sometimes called strongly complete), denoted by $K_{m}$ where $m$ is the number of vertices, if for every two distinct vertices $a$ and $b$ of $\Gamma $, there exist the edges $(a,b)$ and $(b,a)$ in $\Gamma $.
\item[ii.] The digraph $\Gamma $ is called a tournament, denoted by $T_{m}$ where $m$ is the number of vertices, if for every two distinct vertices $a$ and $b$ of $\Gamma $, there is at least one of the two edges $(a,b)$ and $(b,a)$ in $\Gamma $.
\item[iii.] The digraph $\Gamma $ is called connected (sometimes called strongly connected) if for every two distinct vertices $a$ and $b$ of $\Gamma $, there is a directed path from $a$ to $b$ and another from $b$ to $a$.
\item[iv.] The digraph $\Gamma $ is called unilaterally connected if for every two distinct vertices $a$ and $b$ of $\Gamma $, there is a directed path from $a$ to $b$ or from $b$ to $a$.
\item[v.] The digraph $\Gamma $ is called weakly connected if for every two distinct vertices $a$ and $b$ of $\Gamma $, there is a path between $a$ and $b$ without direction.
\item[vi.] Let $a\in V(\Gamma )$. We have the following degrees,
\begin{equation*}
indeg(a):=|\{b\in V(\Gamma )|\text{ }a\neq b\text{ and }(b,a)\in E(\Gamma)\}|\text{ ,}
\end{equation*}
i.e. the number of edges coming to the vertex $a$, and
\begin{equation*}
outdeg(a):=|\{c\in V(\Gamma )|\text{ }a\neq c\text{ and }(a,c)\in E(\Gamma
)\}|\text{ ,}
\end{equation*}
i.e. the number of edges emanating from the vertex $a$. If a vertex $a$ has $l$ loops, then
\begin{equation*}
indegl(a):=l+indeg(a)\text{ \ \ and \ \ }outdegl(a):=l+outdeg(a).
\end{equation*}
\item[vii.] Two digraphs $\Gamma _{1}$ and $\Gamma _{2}$ are said to be isomorphic, denoted by $\Gamma _{1}\cong \Gamma _{2}$, if there is a bijective map $\Phi $ between the vertex set of $\Gamma _{1}$ and the vertex set of $\Gamma _{2}$ such that for any two vertices $a$ and $b$ of $\Gamma_{1}$, $(a,b)$ is an edge in $\Gamma _{1}$ if and only if $(\Phi (a),\Phi (b))$ is an edge in $\Gamma _{2}$.
\item[viii.] The undirected graph $G$ is called complete, denoted by $K_{m}$ where $m$ is the number of vertices, if any two distinct vertices are connected by an edge (possibly with loops).
\item[ix.] The undirected graph $G$ is called connected if there is a path between every two distinct vertices.
\item[x.] Let $a\in V(G)$. We have two ways of counting the degree of this vertex, 
\begin{equation*}
deg(a):=|\{b\in V(G):a\neq b\text{ and }\{a,b\}\in E(G)\}|\text{ ,}
\end{equation*}
i.e. the number of distinct vertices adjacent to $a$. If a vertex $a$ has $l$ loops, then 
\begin{equation*}
degl(a):=l+deg(a).
\end{equation*}
\item[xi.] Two graphs $G_{1}$ and $G_{2}$ are said to be isomorphic, denoted by $G_{1}\cong G_{2}$, if there is a bijective map $\Psi $ between the vertex set of $G_{1}$ and the vertex set of $G_{2}$ such that for any two vertices $a$ and $b$ of $G_{1}$, $\{a,b\}$ is an edge in $G_{1}$ if and only if $\{\Psi (a),\Psi (b)\}$ is an edge in $G_{2}$.\newline
\end{description}
\end{Defns}

\section{Irreducible Divisor Simplicial Complexes}
In \cite{simplicial}. N. Baeth and J. Hobson introduced the concept of an irreducible divisor simplicial complex of any nonzero nonunit element $x$ in a commutative domain $D$. A simplicial complex $S$ is the ordered pair $(V,F)$, where $V$ is the set of vertices and $F$ is the set of faces, where a face is a collection of subsets of $V$ satisfying:\newline
(1) $\{v\}\in F$ for all $v\in V$, and\newline
(2) if $\ \delta \in F$ and $\beta \subseteq \delta $, then $\beta \in F$.\newline
These two conditions mean that any vertex will be considered a face, and any subset of the face is also a face.\newline
A face $\delta \in F$ that is maximal with regard to inclusion is called a facet of $S$. A face $\delta =\{a_{1},a_{2},...a_{d+1}\}$ is said to have dimension $d$. For a nonnegative integer $l$, the $l-$skeleton of $S$, denoted by $SK_{l}$, is the subcomplex of $S$ consisting of all faces of $S$ whose dimension is at most $l$.\\
\newline
For a domain $D$ (not necessarily commutative), we now introduce the notion of an irreducible divisor simplicial complex of any $x\in D^{\#}$.
\begin{definition}
Let $D$ be an atomic domain and $x\in D^{\#}$. The irreducible divisor simplicial complex of $x$, denoted by $S(x)$, is the ordered pair $(V,F)$ with the set of vertices $V=\{y\in \overline{Irr}(D):y\mid x\}$, and the set of faces $F=\{\{y_{1},y_{2},...,y_{n}\}:y_{\sigma (1)}y_{\sigma (2)}...y_{\sigma (n)}\mid x$ for some permutation $\sigma \in S_{n}\}$.\newline
In addition, we place $n-1$ loops on vertex $y$ if $y^{n}\mid x$ but $y^{n+1}\nmid x$, and by convention, we also put $\emptyset \in F$.
\end{definition}
\subsection{Examples}
\begin{enumerate}
\item Let $D=\Bbb{K}[x,y]/<xy-yx-1>$ be the first weyl algebra over a field $\Bbb{K}$.\\
(A) Consider the element $g(x,y)=xy+xy^{2}\in D^{\#}$, the only nonassociate irreducible factorizations of $g(x,y)$ into irreducibles are $xy(1+y)$ and $x(1+y)y$. Therefore, $G(g(x,y))$ is as in Figure \ref{fig:subim1}, and $S(g(x,y))$ is as in Figure \ref{fig:subim2}. 
\begin{figure}[H]
\begin{subfigure}{0.5\textwidth}
 \includegraphics[width=0.9\linewidth, height=4cm]{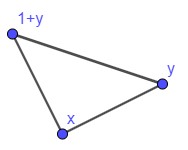} 
\caption{$G(g(x,y))$}
\label{fig:subim1}   
\end{subfigure}
\begin{subfigure}{0.5\textwidth}
 \includegraphics[width=0.9\linewidth, height=4cm]{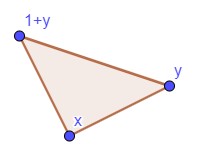}
\caption{$S(g(x,y))$}
\label{fig:subim2}   
\end{subfigure}

\caption{$G(g(x,y))$ and $S(g(x,y))$ in $\Bbb{K}[x,y]/<xy-yx-1>$}
\label{fig:image2}
\end{figure}
Note that in \ref{fig:subim2}, $S(g(x,y))=(V,F)$, with $V=\{x,y,1+y\}$ and $F=\{\emptyset \}\cup F_{0}\cup F_{1}\cup F_{2}$, where $F_{i}$ denotes the set of faces of $S(g(x,y))$ with dimension $i$:\newline
$F_{0}=\{\{x\},\{y\},\{1+y\}\},$\newline
$F_{1}=\{\{x,y\},\{x,1+y\},\{y,1+y\}\},$\newline
$F_{2}=\{\{x,y,1+y\}\}.$\newline
The facet of $S(g(x,y))$ is $\{x,y,1+y\}$.\\
\newline
(B) Consider the element $h(x,y)=xy+xyx\in D^{\#}$, the only nonassociate irreducible factorization of $h(x,y)$ into irreducibles is $xy(1+x)$. Therefore, $G(h(x,y))$ is as in Figure \ref{fig:subim3}, and $S(h(x,y))$ is as in Figure \ref{fig:subim4}.
\begin{figure}[H]
\begin{subfigure}{0.5\textwidth}
\includegraphics[width=0.9\linewidth, height=4cm]{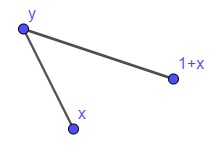} 
\caption{$G(h(x,y))$}
\label{fig:subim3}
\end{subfigure}
\begin{subfigure}{0.5\textwidth}
\includegraphics[width=0.9\linewidth, height=4cm]{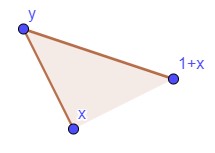}
\caption{$S(h(x,y))$}
\label{fig:subim4}
\end{subfigure}

\caption{$G(h(x,y))$ and $S(h(x,y))$ in $\Bbb{K}[x,y]/<xy-yx-1>$}
\label{fig:image3}
\end{figure}
\item Let $D=\mathbb{Q} \left\langle x,y\right\rangle $ be a free associative algebra in two indeterminates over a field $\mathbb{Q}$. Consider the element $f(x,y)=x^{4}-xyx+5x^{2}\in D^{\#}$, the only nonassociate irreducible factorization of $f(x,y)$ into irreducibles is $x(x^{2}-y+5)x$ . Therefore, $G(f(x,y))$ and $S(f(x,y))$ are as in Figure \ref{fig:image30}.
\begin{figure}[H]
\begin{center}
 \includegraphics[width=0.5\linewidth, height=2cm]{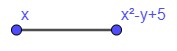}
\caption{$G(f(x,y))$ and $S(f(x,y))$ in $D=\mathbb{Q} \left\langle x,y\right\rangle $}
\label{fig:image30}   
\end{center}
\end{figure}

\item Let $D=H_{\mathbb{Z}}=\mathbb{Z}[1,i,j,k]$ be the Hamilton Quaternion ring over $\mathbb{Z}$.\newline
(A) Consider the element $x=2i+2k\in D^{\#}$. Using norms, we see that the factors into irreducibles are only $\alpha ^{2}\beta $, $\alpha \beta \delta $, $\alpha \beta \alpha $, $-j\beta \alpha ^{2}$, $-k\beta ^{3}$, $\beta \delta ^{2}$, $\delta \alpha \delta $, $i\delta \beta \alpha $, and $j\delta ^{2}\beta $, where $\alpha =1+i$ , $\beta =1+j$, and $\delta =1+k$. Therefore, $G(x)$ is as in Figure \ref{fig:subim5}, and $S(x)$ is as in Figure \ref{fig:subim6}.
\begin{figure}[H]
\begin{subfigure}{0.5\textwidth}
\includegraphics[width=0.9\linewidth, height=4cm]{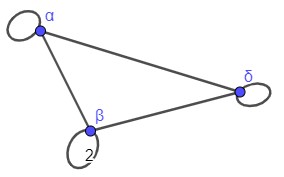} 
\caption{$G(x)$}
\label{fig:subim5}
\end{subfigure}
\begin{subfigure}{0.5\textwidth}
\includegraphics[width=0.9\linewidth, height=4cm]{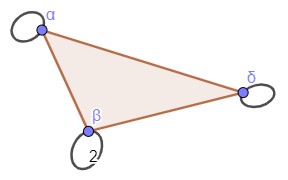}
\caption{$S(x)$}
\label{fig:subim6}
\end{subfigure}
\caption{$G(x)$ and $S(x)$ in $D=\mathbb{Z}[1,i,j,k]$}
\label{fig:image3}
\end{figure}
(B) Consider the element $y=1+i+j+k\in D^{\#}$. Then, again using norms, we see that the only nonassociate irreducible factorizations of $y$ into irreducibles are $(1+i)(1+j)$, $(1+j)(1+k)$ and $(1+k)(1+i)$. Therefore, $G(y)$ and $S(y)$ are as in Figure \ref{fig:image4}.
\begin{figure}[H]
\begin{center}
 \includegraphics[width=0.5\linewidth, height=4cm]{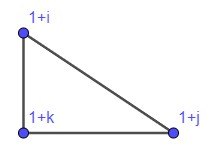}
\caption{$G(y)$ and $S(y)$ in $D=\mathbb{Z}[1,i,j,k]$}
\label{fig:image4}   
\end{center}
\end{figure}
\item Let $R=\mathbb{Z}[ 1,i,j,k]$, and $D=R\left\langle x,y,z\right\rangle $ be a free associative algebra in three indeterminates over the domain $R$.\newline
The only nonassociate irreducible factorizations of $f(x,y,z)=yz+x^{2}yz$ into irreducibles are $(1+ix)(1-ix)yz$, $(1+jx)(1-jx)yz$, and $(1+kx)(1-kx)yz $. Therefore, $G(f(x,y,z))$ is as in Figure \ref{fig:subim7}, and $ S(f(x,y,z))$ is as in Figure \ref{fig:subim8}.
\begin{figure}[H]
\begin{subfigure}{0.5\textwidth}
\includegraphics[width=0.9\linewidth, height=4cm]{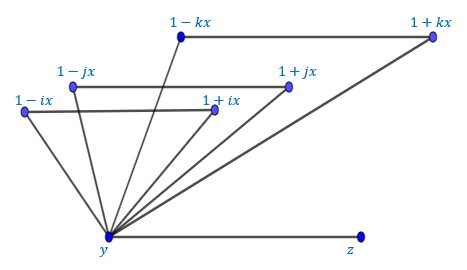} 
\caption{$G(f(x,y,z))$}
\label{fig:subim7}
\end{subfigure}
\begin{subfigure}{0.5\textwidth}
\includegraphics[width=0.9\linewidth, height=4cm]{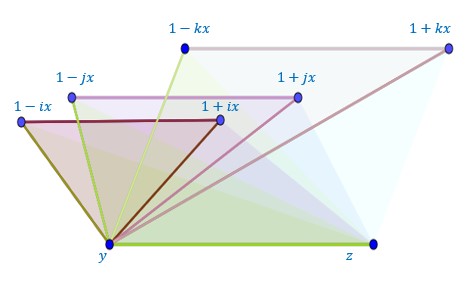}
\caption{$S(f(x,y,z))$}
\label{fig:subim8}
\end{subfigure}
\caption{$G(f(x,y,z))$ and $S(f(x,y,z))$ in $D=R\left\langle x,y,z\right\rangle $}
\label{fig:image3}
\end{figure}
Here we have that, $S(f(x,y,z))=(V,F)$ with $V=\{1\pm ix,1\pm jx,1\pm kx,y,z\}$ and $F=\{\emptyset \}\cup F_{0}\cup F_{1}\cup F_{2}\cup F_{3}$, where\newline
$F_{0}=\{\{1+ix\},\{1-ix\},\{1+jx\},\{1-jx\},\{1+kx\},\{1-kx\},\{y\},\{z\}\}, $\newline
$F_{1}=\{\{1\pm ix\},\{1+ix,y\},\{1-ix,y\},\{1\pm
jx\},\{1+jx,y\},\{1-jx,y\},\{1\pm kx\},\{1+kx,y\},\{1-kx,y\},\{y,z\}\},$\newline
$F_{2}=\{\{1\pm ix,y\},\{1\pm jx,y\},\{1\pm kx,y\},\{1+ix,y,z\},\{1-ix,y,z\},\{1+jx,y,z\},\{1-jx,y,z\},\{1+kx,y,z\},\{1-kx,y,z\}\}$, and\newline
$F_{3}=\{\{1\pm ix,y,z\},\{1\pm jx,y,z\},\{1\pm kx,y,z\}\}.$\newline
The facets of $S(f(x,y,z))$ are $\{1\pm ix,y,z\}$, $\{1\pm jx,y,z\}$, and $\{1\pm kx,y,z\}$.

\item Let $D$ be a UFD and $x$ any nonzero nonunit in $D$. Then we may factor $x$ as $x_{1}x_{2}...x_{m}$ where $x_{1},x_{2},...,x_{m}$ are irreducibles (not necessarily distinct). This is the only way to factor $x$ into irreducibles (unique up to associates of the irreducibles and up to the length of the factorization). Then $G(x)$ is connected, and every set of consecutive distinct irreducibles $\{x_{i},x_{i+1},x_{i+2},...,x_{k}:1\leq i\leq k\leq m\}$ forms a face in $S(x)$.
\end{enumerate}

\subsection{Results}
Let $D$ be a noncommutative atomic domain, and $S(x)=(V,F)$ an irreducible simplicial complex of $x\in D^{\#}$. Clearly, $F$ is a collection of subsets of $V$. First, if $y\in V$ , then $\{y\}\in F$ since $y\mid x$ and hence vertices are faces. Second, suppose that $\delta \in F$ and $\beta \subseteq \delta $, then $\beta $ is not necessarily belonged to $F$ as shown in Figure \ref{fig:subim3} ($\{x,y,1+x\}\in F$ but $\{x,1+x\}\notin F$). Therefore, the irreducible simplicial complexes are not simplicial complexes standard. Whereas if $D$ is an n-atomic domain, then the irreducible simplicial
complexes are simplicial complexes standard.\newline
\begin{remark}
\label{hhh}For an atomic domain $D$ and $x\in D^{\#}$, let $G(x)=(V,E)$ denote the undirected irreducible divisor graph of $x$, and $S(x)=(V^{^{\prime }},F)$ denote the irreducible divisor simplicial complex of $x$. By definition, $V^{^{\prime }}=V=\{y\in \overline{Irr}(D):y\mid x\}$. Furthermore, $E\subseteq F$ since if $\{a,b\}\in E$, we have $ab\mid x$ or $ba\mid x$ and hence $\{a,b\}\in F$. Moreover, if $\{a,b\} $ is a face of $F$ in one dimensional, then $ab\mid x$ or $ba\mid x$, and hence $\{a,b\}\in E$. Hence, the one-dimensional faces of $S(x)$ are precisely the edges of $G(x)$. So we have \ $SK_{1}(S(x))=G(x)$ (i.e. the 1-skeleton of $S(x)$ is precisely $G(x)$).
\end{remark}
Therefore, we see that the concept of the irreducible divisor simplicial complex is a higher dimensional analog of the undirected irreducible divisor graph. Because this structure generally contains components of dimensions two and higher. Therefore, $S(x)$ carries more information than $G(x)$ about the factorizations of the element $x$ in the domain $D$.\newline

The following result shows that it is easier to find factorizations of $x$ by looking at $S(x)$.
\begin{proposition}
\label{normal}For an n-atomic domain $D$ and $x\in D^{\#}$, let $A=\{a_{1},a_{2},...,a_{m}\}$ be a facet of the irreducible divisor simplicial complex $S(x)$. Then there exists a factorization of $x$ such that $\overline{Irr}(x)=A$.
\end{proposition}
\begin{proof}
Since $A$ is a face of $S(x)$, it is clear from the definition of $S(x)$ that
\begin{equation*}
a_{\sigma (1)}a_{\sigma (2)}...a_{\sigma (m)}\mid x,
\end{equation*}
for some permutation $\sigma \in S_{m}$. Thus
\begin{equation*}
x=ba_{\sigma (1)}a_{\sigma (1)}...a_{\sigma (m)}c,
\end{equation*}
where $b,c\in D^{\ast }$. Since $D$ is n-atomic,%
\begin{equation*}
x=b_{1}b_{2}...b_{t}a_{\sigma (1)}a_{\sigma (1)}...a_{\sigma
(m)}c_{1}c_{2}...c_{l},
\end{equation*}
where $b_{i}(i=1,...,t)$ and $c_{j}(j=1,...,l)$ are normal irreducible elements. Now suppose that $b_{i}\notin \{a_{1},a_{2},...,a_{m}\}$ for some $i$ and using $b_{i}$ is normal, then we have 
\begin{equation*}
b_{i}a_{\sigma (1)}a_{\sigma (2)}...a_{\sigma (m)}\mid x.
\end{equation*}
Therefore, $\{b_{i},a_{1},a_{2},...,a_{m}\}$ is a face of $S(x)$ strictly larger than $A$, contradicting that $A$ is a facet of $S(x)$. Hence, $b_{i}\in \{a_{1},a_{2},...,a_{m}\}$ or $b_{i}$ is a unit for $i=1,...,t$. Similarly, $c_{j}\in \{a_{1},a_{2},...,a_{m}\}$ or $c_{j}$ is a unit for $j=1,...,l$. Therefore, $x$ has a factorization such that $\overline{Irr}(x)=A $.
\end{proof}
\bigskip \newline

The following example shows that the hypothesis of $D$ being normal in Proposition \ref{normal} cannot be dropped.

\begin{example}
Let $D=\mathbb{Q}\left\langle x,y\right\rangle $. The only nonassociate irreducible factorization of $f(x,y)=xy^{2}-xy^{2}x$ into irreducibles is $xy^{2}(1-x)$. Therefore, $S(f(x,y))$ is as in Figure \ref{fig:image5}.
\begin{figure}[H]
\begin{center}
 \includegraphics[width=0.5\linewidth, height=4cm]{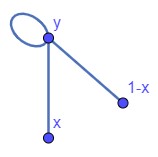}
\caption{$S(f(x,y))$ in $D=\mathbb{Q}\left\langle x,y\right\rangle $}
\label{fig:image5}   
\end{center}
\end{figure}
From Figure \ref{fig:image5}, we see that the facets are $A=\{x,y\}$ and $B=\{y,1-x\}$. However, neither A nor B is $\overline{Irr}(f(x,y))$ since $\overline{Irr}(f(x,y))=\{x,y,(1-x)\}$.\newline
\end{example}

\bigskip

We now produce the main result, which provides the necessary and sufficient condition for an n-atomic domain to be an n-UFD, and get a generalization of Theorem 3.9 in \cite{simplicial} to the noncommutative setting. In the next theorem, we may safely ignore all loops in both $S(X)$ and $G(X)$.\\

In the next theorem, we use the following condition.\\
\newline
 \textbf{ Condition (*)} If $a, r,$ and $r'$ are normal irreducible elements in $D$ such that $ar=r'a$. Then $r$ and $r'$ are associated.\\
\begin{theorem}
\label{1111}Let $D$ be an n-atomic domain.
\begin{enumerate}
\item If $D$ is an n-UFD, then for every $x\in D^{\#}$, $S(x)=(A,P(A))$ for some $A\subseteq \overline{Irr}(x)$, where $P(A)$ is the power set of $A$.
\item If $D$ satisfies Condition (*) and for every $x\in D^{\#}$, $S(x)=(A,P(A))$ for some $A\subseteq \overline{Irr}(x)$, then $D$ is an n-UFD.
\end{enumerate}
\end{theorem}
\begin{proof}
\begin{enumerate}
    \item Let $D$ be an n-UFD, and $x\in D^{\#}$ factors uniquely as 
    \begin{equation*} x=a_{1}^{m_{1}}a_{2}^{m_{2}}...a_{l}^{m_{l}}, \end{equation*}
    where $m_{1},m_{2},...,m_{l}\in \mathbb{N}$ and $a_{1},a_{2},...,a_{l}$ are normal irreducible elements (not necessarily distinct). Then $a_{i_{1}}a_{i_{2}}...a_{i_{t}}\mid x$ $($with $i_{1}<i_{2}<...<i_{t})$ for any subset $\{a_{i_{1}},a_{i_{2}},...,a_{i_{t}}\}\subseteq \overline{Irr}(x)$. Hence $F(S(x))=P(\overline{Irr}(x))$, and $S(x)=(\overline{Irr}(x),P(\overline{Irr}(x))$.
    \item Now let $D$ satisfies Condition (*) and $x\in D^{\#}$, $S(x)=(A,P(A))$ for some $A\subseteq \overline{Irr}(x) $. Since $G(x)=SK_{1}(S(x))$ by Remark \ref{hhh} and $SK_{1}((A,P(A)))$ is a complete graph, $G(x)$ is complete. Consequently, $D$ is an n-UFD by \cite[Theorem 15]{paperphd1}.
\end{enumerate}
\end{proof}

\bigskip

We now examine another necessary and sufficient condition for an n-atomic domain to be an n-UFD. 

\bigskip

The following proposition introduces an equivalent definition of the\ n-UFD.

\begin{proposition}
\label{primeirreducible}An n-atomic domain $D$ is an n-UFD if and only if every normal irreducible element of $D$ is prime.
\end{proposition}
\begin{proof}
Suppose that $D$ is an n-UFD, $r$ is a normal irreducible element in $D$, and $r\mid ba$ for some $a,b\in D^{\ast }$. Then 
\begin{equation}
ba=x^{^{\prime }}rx,  \label{equ}
\end{equation}
where $x^{^{\prime }},x\in D^{\ast }$. Since $D$ is an n-UFD, it follows that $b=b_{1}b_{2}\dots b_{p}$, $a=a_{1}a_{2}\dots a_{m}$, $x^{^{\prime}}=x_{1}^{^{\prime }}x_{2}^{^{\prime }}\dots x_{l}^{^{\prime }}$ and $x=x_{1}x_{2}\dots x_{t}$, where $b_{i},a_{j},x_{k},x_{q}^{^{\prime }},r\ (i=1,\dots ,p,$ $j=1,\dots ,m,$ $k=1,\dots ,t,$ and $q=1,\dots ,l)$ are normal irreducible elements. Replacing in equation (\ref{equ}),
\begin{equation*}
b_{1}b_{2}\dots b_{p}a_{1}a_{2}\dots a_{m}=x_{1}^{^{\prime }}x_{2}^{^{\prime}}\dots x_{l}^{^{\prime }}rx_{1}x_{2}\dots x_{t},
\end{equation*}
then $r$ is associated with some $b_{i}\ (i=1,\dots ,p)$ or with some $a_{j}\ (j=1,\dots ,m)$. Therefore, $r$ divides $b$ or $r$ divides $a$. Thus, $r$ is prime.\newline
Now let every normal irreducible element be prime and $x\in D^{\#}$ such that
\begin{equation}
x=x_{1}x_{2}x_{3}...x_{p}=y_{1}y_{2}y_{3}...y_{m},  \label{eq 2}
\end{equation}
where $x_{i}(i=1,\dots ,p)$ and $y_{j}(j=1,\dots ,m)$ are normal irreducible elements. Then $x_{i}\mid y_{1}y_{2}y_{3}...y_{m}$. Since $x_{i}$ is prime for every $i=1,...,p$, it follows that $x_{i}\mid y_{j}$ for some $j=1,\dots,m$, thus $y_{j}=tx_{i}t^{^{\prime }}$, where $t,t^{^{\prime }}\in D^{\ast }$. Since $x_{i}$ and $y_{j}$ are irreducible elements, we have $t,t^{^{\prime}}\in U(D)$. Therefore, $x_{i}$ $(i=1,\dots ,p)$ is associated with $y_{\sigma (i)}$, for a permutation $\sigma \in S_{m}$. Now we must prove that $p=m$, which we do by induction on $p$. If $p=1,$ then $x_{1}=y_{1}y_{2}y_{3}...y_{m}$. If $m>1,$ without loss of generality, we may assume that $m=2$ , then $x_{1}=y_{1}y_{2}$. Since $x_{1}$ is irreducible, we have $y_{1}$ or $y_{2}$ as units. This contradicts the irreducibility of $y_{1}$ and $y_{2}$, and so $m=1$ when $p=1$. Now assume that $p>1$ and equal length hold for equations of the form (\ref{eq 2}) with fewer than $p$ normal irreducibles on the left-hand side. Now let $x_{1}x_{2}x_{3}...x_{p}=y_{1}y_{2}y_{3}...y_{m}$. Then $x_{p}$ is associated with $y_{j}$ for some $j=1,\dots ,m$. Thus $y_{j}=ux_{p}v$ where $u,v\in U(D) $. So we can write equation (\ref{eq 2}) as
\begin{equation}
x_{1}x_{2}x_{3}...x_{p}=y_{1}y_{2}y_{3}...y_{j-1}ux_{p}vy_{j+1}...y_{m}.
\end{equation}
Using normality and cancelling, we have
\begin{equation}
x_{1}x_{2}x_{3}...x_{p-1}=y_{1}y_{2}y_{3}...y_{j-1}uv^{\prime }y_{j+1}^{\prime }...y_{m}^{\prime },
\end{equation}
\newline
where $v^{\prime }\in U(D)$ and $y_{j+1}^{\prime },y_{j+2}^{\prime},...,y_{m}^{\prime }$ are normal irreducible elements by \cite[Lemma 3]{paperphd1}. Since equal length holds for equations of the form (\ref{eq 2}) with fewer than $p$ normal irreducibles on the left-hand side, $p-1=m-1$. Therefore, $p=m$.
\end{proof}
\begin{corollary}
\cite[Theorem 16.1.12]{equivalentdef} A commutative atomic domain $D$ is an UFD if and only if every irreducible element of $D$ is prime.
\end{corollary}

Recall that for two simplicial complexes $S=(V,F)$ and $T=(W,G)$, their join $S\ast T$ is the simplicial complex with vertex set $V\cup W$ and with face set $\{A\cup B:A\in F,B\in G\}$.

\begin{lemma}
\label{222}Let $D$ be an n-atomic  domain and $a,b\in D^{\#}$. Then $V(S(a))\cup V(S(b))\subseteq V(S(ab))$. Moreover, if $D$ is an n-UFD, then equality holds.
\end{lemma}
\begin{proof}
Let $x\in V(S(a))\cup V(S(b))$. Then $x\in V(S(a))$ or $x\in V(S(b))$ and $x\mid a$ or $x\mid b$. In either case, $x\mid ab$ and $x\in V(S(ab))$. Now suppose that $D$ is an n-UFD and $x\in V(S(ab))$, then $x\mid ab$ since $x$ is normal irreducible and hence prime by Proposition \ref{primeirreducible}. Hence $x\mid a$ , $x\in V(S(a))$ or $x\mid b$ , $x\in V(S(b))$. Thus $x\in V(S(a))\cup V(S(b))$.
\end{proof}

\bigskip

We now give another characterization of n-unique factorization domains via irreducible simplicial complexes. The following theorem extends Theorem 3.12 in \cite{simplicial} to the noncommutative domain.
\begin{theorem}
Let $D$ be an n-atomic domain. Then the following are equivalent:
\begin{enumerate}
\item $D$ is an n-UFD.
\item $S(a)\ast S(b)=S(ab)$ for every $a,b\in D^{\#}$.
\end{enumerate}
\end{theorem}
\begin{proof}
Let $D$ be an n-UFD and $a,b\in D^{\#}$. From Lemma \ref{222}, we get $V(S(a)\ast S(b))=V(S(ab))$. From Theorem \ref{1111}, $F(S(x))=P(V(S(x)))$ for any $x\in D^{\#}$ and from Lemma \ref{222}, $V(S(ab))=V(S(a))\cup V(S(b))$. So we have $F(S(ab))=P(V(S(ab)))=P(V(S(a))\cup V(S(b)))=P(V(S(a)\ast S(b)))=F(S(a)\ast S(b))$.\newline
Now let $D$ be not an n-UFD. Using Proposition \ref{primeirreducible}, there exists a normal irreducible element $y\in D^{\#}$ that is not prime. Hence there exist $a,b\in D^{\#}$ such that $y\mid ab$, whereas $y\nmid a$ and $y\nmid b$. Thus $y\in V(S(ab))$ but\ $y\notin V(S(a))\cup V(S(b))=V(S(a)\ast S(b))$, then $S(a)\ast S(b)\neq S(ab)$.
\end{proof}

\section{Generalized Irreducible Divisor Graphs}
The goal of this section is to introduce the notions of a noncommutative $\tau $-factorization and a $\tau $-irreducible divisor graph in a noncommutative domain $D$. More precisely, we will define the directed and undirected $\tau $-irreducible divisor graphs for nonzero nonunit elements in a noncommutative domain, and we will find an equivalent characterization of an n-UFD.\newline
We begin with some definitions for a noncommutative $\tau $-factorization.
\subsection{$\protect\tau $- Factorization Definitions}
Let $D$ be a domain with a symmetric relation $\tau $ on $D^{\#}$. D. D. Anderson and A. M. Frazier in \cite{tauirreducible} introduced the following definitions for a $\tau $-factorization of an element $a\in D^{\#}$ in the commutative case, and we will use the same definitions in the noncommutative case.
\begin{Defns}
\begin{enumerate}
\item A factorization of $a\in D^{\#}$, $a=\lambda a_{1}a_{2}...a_{m}$ is called a $\tau $-factorization if $a_{i}\in D^{\#}$, $\lambda \in U(D)$ and $a_{i}\tau a_{j}$ for every $i,j=1,...,m$. If $m=1$, then this is called a trivial $\tau $-factorization, each $a_{i}$ is called a $\tau $-factor, or $a_{i}$ $\tau $-divides $a$, written $a_{i}\mid _{\tau }a$.
\item A relation $\tau $ is said to be associate preserving if for $a,b,b^{^{\prime }}\in D^{\#}$, with $a\tau b$ and $b\sim b^{^{\prime }}$ imply $a\tau b^{^{\prime }}$.
\end{enumerate}
\end{Defns}

\begin{examples}
\begin{enumerate}
\item Let $D$ be a domain, and $\tau =D^{\#}\times D^{\#}$. Then the usual factorization (resp. usual divides) is the same as the $\tau $-factorization (resp. $\tau $-divides).
\item Let $D$ be a domain, and $\tau =\varnothing $. For every $a\in D^{\#}$, there is only trivial factorization. Furthermore, all $\tau $-divisors of $a$ are associated with $a$.
\item Let $D$ be a domain, and $S$ be a non-empty subset of $D^{\#}$. Define $a\tau b$ if and only if $a,b\in S$, i.e. $\tau =S\times S$. A non-trivial $\tau $-factorization is a factorization into elements from $S$. For example, if $S$ is the set of primes (resp. irreducibles), then the $\tau $-factorization is a prime decomposition (resp. an atomic factorization).
\end{enumerate}
\end{examples}
We now define the noncommutative $\tau $-irreducible, $\tau $-ascending chain condition on the principal right (left) ideals, and $\tau $-unique factorization domain.
\begin{Defns}
\begin{enumerate}
\item Let $x\in D^{\#}$. We say that $x$ is a $\tau $-irreducible or a $\tau 
$-atom if the factorization of the form $x=\lambda (\lambda ^{-1}xv)v^{-1}$is the only $\tau $-factorizations of $x$.
\item We say that a domain $D$ satisfies the $\tau $-ascending chain condition on the principal right (left) ideals ($\tau $-ACCPr) if for every chain $\left\langle a_{0}\right\rangle \subseteq \left\langle a_{1}\right\rangle \subseteq $\textperiodcentered \textperiodcentered \textperiodcentered $\subseteq \left\langle a_{i}\right\rangle \subseteq $\textperiodcentered \textperiodcentered \textperiodcentered\ with $a_{i+1}|_{\tau }a_{i}$, there exists an $n\in \mathbb{N}$ such that $\left\langle a_{n}\right\rangle =\left\langle a_{j}\right\rangle $ for all $j>n$.
\item A domain $D$ is called $\tau $-atomic (resp. $\tau $-normal atomic, for short $\tau $-n-atomic) if for every $r\in D^{\#}$ there exist $\tau $-irreducible (resp. normal $\tau $-irreducible) elements $r_{1},r_{2},\dots ,r_{l}$ such that $r=r_{1}r_{2}\dots r_{l}$.
\item A domain $D$ is called a $\tau $-finite factorization domain, for short $\tau $-FFD (resp. $\tau $-normal finite factorization domain, for short $\tau $-n-FFD), if $:$\\
    (1) the domain $D$ is $\tau $-atomic (resp. $\tau $-n-atomic), and\\ 
    (2) if $r\in D^{\#}$ has only finitely many distinct nonassociate $\tau $-irreducible (resp. normal $\tau $-irreducible) divisors.
\item A domain $D$ is called a $\tau $-unique factorization domain, for short $\tau $-UFD (resp. $\tau $-normal unique factorization domain, for short $\tau $-n-UFD), if $:$\\ 
    (1) the domain $D$ is $\tau $-atomic (resp. $\tau $-n-atomic), and\\ 
    (2) if $r\in D^{\#}$ has two $\tau $-atomic (resp. $\tau $-n-atomic) factorizations, i.e. $r=r_{1}r_{2}\dots .r_{l}=t_{1}t_{2}\dots .t_{m}$ \ where $r_{1},r_{2},\dots,r_{l},t_{1},t_{2},\dots ,t_{m}$ are $\tau $-irreducible (resp. normal $\tau $-irreducible) elements, then $l=m$ and there is a permutation $\sigma \in S_{l}$ such that for each $i=1,\dots ,l,$ $r_{i}$ is associated with $t_{\sigma (i)}$.
\end{enumerate}
\end{Defns}

\subsection{$\protect\tau $-Irreducible Divisor Graphs}
Let $D$ be a domain with a symmetric and associate preserving relation $\tau $ on $D^{\#}$. Now, we introduce the notions of directed and undirected $\tau $-irreducible divisor graphs for any $x\in D^{\#}$ in a noncommutative
domain $D$.
\begin{definition}
\label{Def t}Let $D$ be a $\tau $-atomic domain with a symmetric and associate preserving relation $\tau $ on $D^{\#}$and $x\in D^{\#}$. The directed (resp. undirected) $\tau $-irreducible divisor graph of $x$, denoted by $\Gamma _{\tau }(x)$ (resp. $G_{\tau }(x)$), is the graph $(V,E)$ with the set of vertices $V=\left\{ y\in \overline{Irr_{\tau }}(D):y\mid_{\tau }x\right\} $, and the set of edges $E=\{(y_{1},y_{2}):y_{1}y_{2}\mid _{\tau }x\}$ (resp. $E=\{\{y_{1},y_{2}\}:y_{1}y_{2}\mid _{\tau }x \ or \ y_{2}y_{1}\mid _{\tau }x\}$).\newline
Further, we attach $n-1$ loops to the vertex $y$ if $y^{n}\mid _{\tau }x$ and $y^{n+1}\nmid _{\tau }x$.\newline
The $\tau $-reduced directed (resp.\ undirected) divisor graph of $x$\ in $D^{\#}$ is the subgraph of $\Gamma _{\tau }(x)$ (resp. $G_{\tau }(x)$) containing no loops and denoted by $\overline{\Gamma _{\tau }(x)}$ (resp. $\overline{G_{\tau }(x)}$).
\end{definition}
Note that the definition of the undirected $\tau $-irreducible divisor graphs of $x\in D^{\#}$ coincides with the definition of the commutative $\tau $-irreducible divisor graphs of $x\in D^{\#}$ in \cite{generalirreducible}.

\begin{examples}
\begin{enumerate}
\item Let $D$ be a domain and $\tau =\varnothing $. In this case, every nonzero nonunit is $\tau $-irreducible. This means for every $x\in D^{\#}$, $\Gamma _{\tau }(x)=G_{\tau }(x)=(\{x\},\varnothing )$.
\item Let $D$ be a domain and $\tau =D^{\#}\times D^{\#}$. In this case, every $\tau $-factorization is a usual factorization and conversely. Moreover, an element $x\in D^{\#}$ is $\tau $-irreducible if and only if $x$ is irreducible. Hence we have $G_{\tau }(x)=G(x)$, and $\Gamma _{\tau }(x)=\Gamma (x)$.
\item Let $D=\mathbb{Q}\left\langle x,y\right\rangle $ and the relation $\tau $ defined by $h(x,y)\tau g(x,y)$ if and only if $\deg (h(x,y))=\deg (g(x,y))$. Consider the element $f(x,y)=x^{4}-xyx+5x^{2}$. The only factorization of $f(x,y)$ into nonassociate irreducibles is $x(x^{2}-y+5)x$, and it is\ not a $\tau $-factorization, so $f(x,y)$ is a $\tau $-irreducible element. Therefore, $G(f(x,y))$ is as in Figure \ref{fig:subim14}, $\Gamma(f(x,y))$ is as in Figure \ref{fig:subim13}, and $G_{\tau }(f(x,y))=\Gamma_{\tau }(f(x,y))$ is as in Figure \ref{fig:subim15}.
\begin{figure}[H]
\begin{subfigure}{0.3\textwidth}
\includegraphics[width=0.9\linewidth, height=4cm]{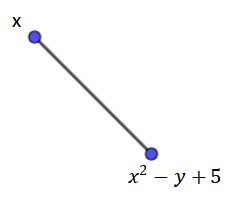} 
\caption{$G(h(x,y))$}
\label{fig:subim14}
\end{subfigure}
\begin{subfigure}{0.3\textwidth}
\includegraphics[width=0.9\linewidth, height=4cm]{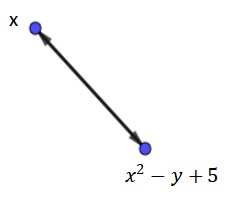} 
\caption{$\Gamma(f(x,y))$}
\label{fig:subim13}
\end{subfigure}
\begin{subfigure}{0.3\textwidth}
\includegraphics[width=0.9\linewidth, height=4cm]{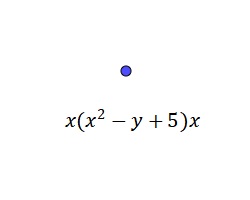}
\caption{$G_{\tau }(f(x,y))=\Gamma_{\tau }(f(x,y))$}
\label{fig:subim15}
\end{subfigure}

\caption{$\Gamma(f(x,y))$, $G(f(x,y))$, and $G_{\tau }(f(x,y))$ in $\mathbb{Q}\left\langle x,y\right\rangle $}
\label{fig:image3}
\end{figure}

\item Let $D=K[x,y]\diagup \left\langle xy-yx-1\right\rangle $ be the first weyl algebra over a field $k$ and the relation $\tau $ defined by $h(x,y)\tau g(x,y)$ if and only if $\deg (h(x,y))=\deg (g(x,y))$. Let $f(x,y)=xy+xy^{2}$. The only factorizations of $f(x,y)$ into nonassociate irreducibles are $xy(1+y)$ and $x(1+y)y$, both are $\tau $-factorizations. Therefore, $G(f(x,y))$ and $G_{\tau }(f(x,y))$ are as in Figure \ref{fig:subim9} whereas $\Gamma (f(x,y))$ and $\Gamma _{\tau }(f(x,y))$ are as in Figure \ref{fig:subim10}. 
\begin{figure}[H]
\begin{subfigure}{0.5\textwidth}
\includegraphics[width=0.9\linewidth, height=4cm]{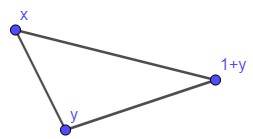} 
\caption{$G(f(x,y))=G_{\tau }(f(x,y))$}
\label{fig:subim9}
\end{subfigure}
\begin{subfigure}{0.5\textwidth}
\includegraphics[width=0.9\linewidth, height=4cm]{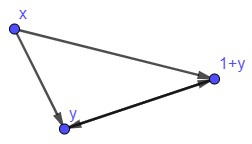}
\caption{$\Gamma (f(x,y))=\Gamma _{\tau }(f(x,y))$}
\label{fig:subim10}
\end{subfigure}
\caption{$G_{\tau }(f(x,y))$ and $\Gamma _{\tau }(f(x,y))$ in $K[x,y]\diagup \left\langle xy-yx-1\right\rangle $ }
\label{fig:image3}
\end{figure}

\end{enumerate}
\end{examples}

\subsection{Results}
In this subsection, we extend the results in \cite{generalirreducible} to both the digraph and undirected graphs of $x\in D^{\#}$ in a noncommutative domain $D$.

\begin{proposition}
Let $D$ be a domain with a symmetric and associate preserving relation $\tau $ on $D^{\#}$. If $D$ is $\tau $-atomic, then a nonunit $x\in D^{\#}$ is $\tau $-irreducible if and only if $G_{\tau }(x)=\Gamma _{\tau}(x)=(\{x\},\varnothing )$, is a single vertex (with no loop).
\end{proposition}
\begin{proof}
$(\Longrightarrow )$ Clear.\\
$(\Longleftarrow )$ Let $x\in D^{\#}$ such that $G_{\tau }(x)$ is a single vertex and $x$ is not $\tau $-irreducible. Then there is a non-trivial $\tau $-atomic factorization $x=a_{1}a_{2}...a_{n}$ with $n\geq 2$. This yields $a_{1},a_{2}\in V(G_{\tau }(x))$, but there is only one vertex and no loops in $G_{\tau }(x)$. Contradicting the hypothesis, $x$ is $\tau $-irreducible.
\end{proof}

\bigskip
Now, we present definitions that will be used in the next results.
\begin{Defns}
Let $D$ be a domain with a relation $\tau $ on $D^{\#}.$
\begin{enumerate}
\item We say that $\tau $ is right multiplicative if $a,b\in D^{\#}$ such that $a\tau b$, then $ax\tau bx$ for all $x\in D^{\#}$. Similarly, we can define left multiplicative. A relation $\tau $ is multiplicative if it is both right and left multiplicative.
\item We say that $\tau $ is right cancellative if $a,b,x\in D^{\#}$ such that $ax\tau bx$, then $a\tau b$. Similarly, we can define left cancellative. A relation $\tau $ is cancellative if it is both right and left cancellative.
\item We define a $\tau $-refinement of a $\tau $-factorization $a_{1}a_{2}...a_{n}$ to be a factorization of the form $b_{11}...b_{1m_{1}}b_{21}...b_{2m_{2}}...b_{n1}...b_{nm_{n}}$, where $a_{i}=b_{i1}...b_{im_{i}}$ is a $\tau $-factorization for each $i$. We say that $\tau $ is refinable if every $\tau $-refinement of a $\tau $-factorization is a $\tau $-factorization.
\end{enumerate}
\end{Defns}

\begin{lemma}
\label{Lemm t}Let $D$ be a domain with a symmetric and associate preserving relation $\tau $ on $D^{\#}$ and $x$ a normal $\tau $-irreducible element in $D$. If $r,r^{\prime }\in D$ such that $xr=r^{\prime }x$. Then
\begin{enumerate}
    \item $r$ is a unit if and only if $r^{\prime }$ is a unit,
    \item $r$ is a normal element if and only if $r^{\prime }$ is a normal element, and
    \item if $\tau $ is right multiplicative and left cancellative. Then $r$ is $\tau $-irreducible if and only if $r^{\prime }$ is $\tau $-irreducible.
\end{enumerate}
\end{lemma}

\begin{proof}
We will suffice with proving "if" and proving "only if" will be similar.
\begin{enumerate}
    \item Let $r$ be a unit and $xr=r^{\prime }x$. Then $x=r^{\prime }xr^{-1}=r^{\prime }r^{-1\prime }x$. Therefore, $r^{\prime }r^{-1\prime }=1$, and $r^{\prime }$ is a right unit. Since $D$ is a domain, $r^{\prime }$ is a unit.
    \item Since $x$ and $r$ are normal elements and $r^{\prime }x=xr$, we have $r^{\prime}Rx=Rxr=Rr^{\prime }x$. Since $D$ is a domain, $r^{\prime}R=Rr^{\prime }$. Therefore, $r^{\prime }$ is normal.
    \item Let $r$ be $\tau $-irreducible$\ $and $r^{\prime }$ not $\ \tau $-irreducible. Then $r^{\prime }=yz$ such that $y,z\in D^{\#}$, and $y\tau z$. Hence \begin{equation} 
    xr=r^{\prime }x=yzx=yxz^{\prime }=xy^{\prime }z^{\prime }.  \label{can mul} 
    \end{equation}
    \newline 
    Therefore, $r=y^{\prime }z^{\prime }$, and from (1), we have that $y^{\prime}$ and $z^{\prime }$ are nonunits. Since $y\tau z$ and $\tau $ is right multiplicative, we have $yx\tau zx$ and from equation (\ref{can mul}), $xy^{\prime }\tau xz^{\prime }$. Since $\tau $ is  left cancellative, we have $y^{\prime }\tau z^{\prime }$. Thus $r$ is not $\tau $-irreducible and this is a contradiction.
\end{enumerate}
\end{proof}

\bigskip

The following examples show that the conditions in part 3 of Lemma \ref{Lemm t} are not redundant.

\begin{example}
\label{not red 1}Let $R=\mathbb{Z}\lbrack 1,i,j,k]$ and $D=R[x]$. The element $(1+i)$ is normal in $D$ because for any element 
$\displaystyle\sum_{m=0}^{n}(a_{0}+a_{1}i+a_{2}j+a_{3}k)_{m}x^{m} \in D$, there is an element $\displaystyle\sum_{m=0}^{n}(a_{0}+a_{1}i-a_{3}j+a_{2}k)_{m}x^{m}\in D$ such that 
\begin{equation*}
(1+i)\displaystyle\sum_{m=0}^{n}(a_{0}+a_{1}i+a_{2}j+a_{3}k)_{m}x^{m}=\displaystyle\sum_{m=0}^{n}(a_{0}+a_{1}i-a_{3}j+a_{2}k)_{m}x^{m}(1+i).
\end{equation*}
Consider the relation $\tau _{1}=\{(\beta f(x),\gamma g(x)), (\gamma g(x),\beta f(x)):f(x),g(x)\in D\}$, where
\begin{equation*}
\beta \in \overline{(1+x)} =\{\pm (1+x),\pm i(1+x),\pm j(1+x),\pm k(1+x)\},
\end{equation*}
\begin{equation*}
\gamma \in \overline{(1+j)} = \{\pm (1+j),\pm (i+k),\pm (j-1),\pm (k-i)\}.
\end{equation*}
It is clear that $\tau _{1}$ is symmetric and right multiplicative. Also, $\tau _{1}$ is
associate preserving, since for any element $\beta \in \overline{(1+x)}$,
\begin{equation*}
\overline{\beta f(x)}=\{u\beta f(x)v:u,v\in U(D)\}=\{u\beta vv^{-1}f(x)v:u,v\in U(D)\},
\end{equation*}
\begin{equation*}
=\{u\beta v {f}^{\prime} (x):u,v\in U(D)\}=\{{\beta}^{\prime} {f}^{\prime}(x)\},
\end{equation*} 
where ${f}^{\prime}(x)=v^{-1}f(x)v \in D$ and ${\beta}^{\prime}=u\beta v \in \overline{(1+x)}$. Similarly, if $\gamma \in \overline{(1+j)}$, then
\begin{equation*}
\overline{\gamma g(x)}=\{{\gamma}^{\prime} {g}^{\prime}(x)\},
\end{equation*} 
 where ${g}^{\prime}(x) \in D$ and ${\gamma}^{\prime} \in \overline{(1+j)}$  . It follows that $\tau _{1}$ is associate preserving. We claim that $\tau _{1}$ is not left cancellative. On the contrary suppose that $\tau _{1}$ is left cancellative.\\
 Since $((1+j)(1+i),(1+x)(1+i))\in \tau _{1}$ and $(1+j)(1+i)=(1+i)(1-k)$ and $(1+x)(1+i)=(1+i)(1+x)$, we have
\begin{equation*}
((1+i)(1-k),(1+i)(1+x))\in \tau _{1},
\end{equation*}
then, 
\begin{equation*}
((1-k),(1+x))\in \tau _{1}.
\end{equation*}%
\newline
This is a contradiction from the definition of $\tau _{1}$. Thus $\tau _{1}$ is not left cancellative. Now, let $y=(1-k)(1+x)$ and $y^{\prime}=(1+j)(1+x)$ such that
\begin{equation*}
(1+i)y=(1+i)(1-k)(1+x)=(1+j)(1+x)(1+i)=y^{\prime}(1+i).
\end{equation*}
Since $(1-k,1+x)\notin \tau _{1}$ and $(1+j,1+x)\in \tau _{1}$, the element $y=(1-k)(1+x)$ is $\tau
_{1}-$irreducible, but $y^{\prime}=(1+j)(1+x)$ is not $\tau _{1}-$irreducible.
\end{example}

\begin{example}
Let $D$ be the same in Example \ref{not red 1} and $\tau _{2}=\{(\beta,\gamma),(\gamma,\beta)\}$, where $\beta \in \overline{(1+x)}$ and $\gamma \in \overline{(1+j)}$. Clear $\tau _{2}$ is symmetric, associate preserving, and left cancellative. However, $\tau_{2}$ is not right multiplicative, because $((1+x),(1+j))\in \tau _{2}$ but $((1+x)x,(1+j)x)\notin \tau _{2}$. In this case, the element $y=(1-k)(1+x)$ is $\tau
_{2}-$irreducible, but $y^{\prime}=(1+j)(1+x)$ is not $\tau _{2}-$irreducible..
\end{example}
The following result gives necessary and sufficient graph theoretic conditions for a $\tau $-n-atomic domain to be a $\tau $-n-FFD.
\begin{theorem}\label{ffd}
Let $D$ be a $\tau $-n-atomic domain. Then the following statements are equivalent:
\begin{enumerate}
\item $D$ is a $\tau $-n-FFD.
\item $\Gamma_{\tau }(x)$ is finite for all $x\in D^{\#}$.
\item For all $x\in D^{\#}$, outdeg$(w)$ and indeg$(w)$ are finite for all $w\in V(\Gamma_{\tau }(x))$.
\item For all $x\in D^{\#}$, outdegl$(w)$ and indegl$(w)$ are finite for all $w\in V(\Gamma_{\tau }(x))$.
\end{enumerate}
\end{theorem}
\begin{proof} 
$1 \Rightarrow 2 \Rightarrow 3$ and $4 \Rightarrow 3$ 
Clear.\\
$3 \Rightarrow 1$ Let $D$ not be $\tau $-n-FFD. Then there exists $y\in D^{\#}$ such that the set of its nonassociate $\tau $-normal irreducible divisors, $A=\{\pi_{i}\}$, is infinite. Thus in $\Gamma_{\tau }(y^{2})$, the vertices $\pi_{j}$ and $\pi_{k}$ are connected by edges $[\pi_{j},\pi_{k}]$ and  $[\pi_{k},\pi_{j}]$ for all $\pi_{j},\pi_{k}\in A$. Therefore, outdeg$(\pi_{j})$ and indeg$(\pi_{j})$ are infinite and $(3)$ fails.\\
$3 \Rightarrow 4$ Let $(4)$ is fails. Then either $(3)$ fails or some vertex $\pi_{j}$ in $\Gamma_{\tau }(x)$ has infinitely loops. In this case, ${\pi_{j}}^\infty$ $\tau $-divides $x$ and it has no mathematical meaning.
\end{proof}\\
\newline
Using the same step as in the proof of Theorem \ref{ffd}, we can derive the following result for an undirected irreducible divisor graph.
\begin{theorem}
Let $D$ be a $\tau $-n-atomic domain. Then the following statements are equivalent:
\begin{enumerate}
\item $D$ is a $\tau $-n-FFD.
\item $G_{\tau }(x)$ is finite for all $x\in D^{\#}$.
\item For all $x\in D^{\#}$, deg$(w)$ is finite for all $w\in V(G_{\tau }(x))$.
\item For all $x\in D^{\#}$, degl$(w)$ is finite for all $w\in V(G_{\tau }(x))$.
\end{enumerate}
\end{theorem}
 The following proposition gives the condition under which the $\tau $- n-atomic domain satisfies $\tau $-ACC on principal right (left) ideals.
\begin{proposition}
Let $D$ be a domain with a symmetric, refinable, associate preserving, right multiplicative, and left cancelitive relation $\tau $ on $D^{\#}$. If $D$ is $\tau $-n- atomic such that for all $x\in D^{\#}$, $outdegl(a)<\infty $ (resp. $indegl(a)<\infty $) for all $a\in V(\Gamma _{\tau }(x))$, then $D$ satisfies the $\tau $-ascending chain condition on the principal left (resp. right) ideals ($\tau $-ACCPr).
\end{proposition}
\begin{proof} 
We will suffice with proving "left" and proving "right" will be similar. Assume that $D$ does not satisfy the $\tau $-ACC on the principal left ideals, then there exists an infinite chain of principal left ideals $\left\langle x_{1}\right\rangle \subsetneq \left\langle x_{2}\right\rangle \subsetneq ...$\ such that $x_{i+1}|_{\tau }x_{i}$. Thus
\begin{equation*}
x_{1}=a_{1}x_{2}=a_{1}a_{2}x_{3}=a_{1}a_{2}a_{3}x_{4}=...
\end{equation*}%
for some $a_{i}\in D^{\#}$. Since $D$ is $\tau $-n-atomic and $\tau $ is refinable,
\begin{equation}
x_{1}=(\displaystyle\prod_{k_{1}=1}^{n_{1}}a_{1,k_{1}})x_{2}=(\displaystyle\prod_{k_{1}=1}^{n_{1}}a_{1,k_{1}})(\displaystyle\prod_{k_{2}=1}^{n_{2}}a_{2,k_{2}})x_{3}=(\displaystyle\prod_{k_{1}=1}^{n_{1}}a_{1,k_{1}})(\displaystyle\prod_{k_{2}=1}^{n_{2}}a_{2,k_{2}})(\displaystyle\prod_{k_{3}=1}^{n_{3}}a_{3,k_{3}})x_{4}=...,  \label{33}
\end{equation}
\newline
where $a_{i,j}$ are $\tau $-normal irreducibles and the factorization in each iteration of equation (\ref{33}) increases in lenght. If the elements $a_{i,j}$ are infinite. Then by the normality of $a_{i,j}$, we have infinite outdegree in $\Gamma _{\tau }(x_{1})$. Otherwise, if $a_{i,j}$ are finite, then one of the $a_{i_{0},j_{0}}$ for some $i_{0}$ and $j_{0}$ appears infinitely often in the $\tau $-factorization of $x_{1}$, and thus $a_{i_{0},j_{0}}$ has an infinite number of loops in $\Gamma _{\tau }(x_{1})$. Either of these conditions implies that $outdegl(a)$ is infinite for some vertex $a$ of $\Gamma _{\tau }(x_{1})$. This is a contradiction, and as desired, $D$ must satisfy the $\tau $-ascending chain condition on the principal left ideals ($\tau $-ACCPr).
\end{proof}

\bigskip

By using the same steps as in the directed $\tau -$irreducible divisor graph, we obtain the setting of the undirected $\tau -$irreducible divisor graph as follows. 

\begin{proposition}
Let $D$ be a domain with a symmetric, refinable, associate preserving, right multiplicative, and left cancelitive relation $\tau $ on $D^{\#}$. If $D$ is $\tau $-n- atomic such that for all $x\in D^{\#}$, $outdegl(a)<\infty $ (resp. $indegl(a)<\infty $) for all $a\in V(G _{\tau }(x))$, then $D$ satisfies the $\tau $-ascending chain condition on the principal left (resp. right) ideals ($\tau $-ACCPr).
\end{proposition}
In the next theorem, we use the following condition.\\
\newline
 \textbf{ Condition (**)} If $a, r,$ and $r'$ are normal $\tau $-irreducible elements in $D$ such that $ar=r'a$. Then $r$ and $r'$ are associated.\\
\newline
The following theorem is the main result of this section and provides the necessary and sufficient conditions for a $\tau $-n-atomic domain to be a $\tau $-n-UFD.

\begin{theorem}
\label{main t}Let $D$ be a $\tau $-n-atomic domain with a symmetric, associate preserving, right multiplicative, and left cancelitive relation $\tau $ on $D^{\#}$.  Consider the following statements:
\begin{enumerate}
\item $D$ is a $\tau $-n-UFD;
\item $\Gamma _{\tau }(x)$ is a tournament for all $x\in D^{\#}$;
\item $\Gamma _{\tau }(x)$ is unilaterally connected for all $x\in D^{\#}$;
\item $\Gamma _{\tau }(x)$ is weakly connected for all $x\in D^{\#}$.
\end{enumerate}
Then
 \begin{center}
  $1 \Rightarrow 2 \Rightarrow 3 \Rightarrow 4$.   
 \end{center}
Moreover, if $D$ satisfies Condition (**), then all statements are equivalent.
\end{theorem}

\begin{proof}
The following proof is a modification of the proof of Theorem 13 in \cite{paperphd1}.\\
$1 \Rightarrow 2$ Let $D$ be a $\tau $-n- UFD and $x$ any nonzero nonunit. Then we may factor $x$ as $x_{1}^{a_{1}}x_{2}^{a_{2}}...x_{l}^{a_{l}}$ where $a_{1},a_{2},...,a_{l}\in \mathbb{N}$ and $x_{1},x_{2},...,x_{l}$ are $\tau $-normal irreducibles (not necessary distinct). Since this is the only way to factor $x$ into $\tau $-normal irreducibles, we see that for every pair of distinct vertices $x_{i}(i=1,...,l)$ in $\Gamma _{\tau }(x)$, there is at least one edge. It
follows that $\Gamma _{\tau }(x)$ is a tournament.\newline
$2 \Rightarrow 3 \Rightarrow 4$ Clear.\newline
Now, we want to prove that $4 \Rightarrow 1$ if $D$ satisfies Condition (**). We show that the set $A$ of all $x\in D^{\#}$ that have at least two distinct $\tau $-factorization into $\tau $-normal irreducibles is empty. Assume otherwise and let $m:=\underset{z\in A}{\min }\{k:z=\pi _{1}\pi _{2}...\pi _{k}$ with $\pi _{i}$ $\tau $-normal irreducible for every $i\}$ clear $m\geq 2$. Thus there exists $y\in D^{\#}$ such that $y=\pi _{1}\pi _{2}...\pi _{m}$. Since $y\in A$, we have another (distinct) $\tau $-normal irreducible factorization of $y=\alpha _{1}\alpha_{2}...\alpha _{t}$ with each $\alpha _{j}(j=1,...,t)$ $\tau $-normal irreducible and $t\geq m$. We claim that each $\pi _{i}$ is nonassociate to each $\alpha _{j}$. Otherwise, if $\pi _{i}$ is associate to $\alpha _{j}$, then 
\begin{equation*}
y=\pi _{1}\pi _{2}...\pi _{i-1}u\alpha _{j}v\pi _{i+1}...\pi _{m}=\alpha_{1}\alpha _{2}...\alpha _{j-1}\alpha _{j}\alpha _{j+1}...\alpha _{t}.
\end{equation*}
Since $\alpha _{j}$ is normal, we have
\begin{equation*}
y=\alpha _{j}\pi _{1}^{\prime }\pi _{2}^{\prime }...\pi _{i-1}^{\prime}u^{\prime }v\pi _{i+1}...\pi _{m}=\alpha _{j}\alpha _{1}^{\prime }\alpha_{2}^{\prime }...\alpha _{j-1}^{\prime }\alpha _{j+1}...\alpha _{t}.
\end{equation*}
So we have an element, that can be denoted by $\frac{y}{\alpha _{j}}$, such that
\begin{equation}
\frac{y}{\alpha _{j}}=\pi _{1}^{\prime }\pi _{2}^{\prime }...\pi_{i-1}^{\prime }u^{\prime }v\pi _{i+1}...\pi _{m}=\alpha _{1}^{\prime}\alpha _{2}^{\prime }...\alpha _{j-1}^{\prime }\alpha _{j+1}...\alpha _{t},
\label{EQ}
\end{equation}
where $u,v,$ and $u^{\prime }$are units and $\pi _{p}^{\prime }(p=1,...,i-1)$ and $\alpha _{n}^{\prime }(n=1,...,j-1)$ are normal $\tau $-irreducible elements by Lemma \ref{Lemm t}. Therefore equation (\ref{EQ}) gives two distinct $\tau $-factorizations of $\frac{y}{\alpha _{j}}$ into $\tau $-normal irreducibles because $D$ satisfies Condition (**), contradicting the minimality of $m$. Thus $\pi _{i}$ is not an associate of any $\alpha _{j}$. Since $\Gamma _{\tau }(y)$ is weakly connected, it implies that there is an edge connecting $\pi _{k}$ and $\alpha _{l}$ for some $k$ and $l$. From Definition \ref{Def t}, we have $\pi_{k}\alpha _{l}\mid _{\tau }y$ or $\alpha _{l}\pi _{k}\mid _{\tau }y$. If $\pi _{k}\alpha _{l}\mid _{\tau }y$, then $y=x\pi _{k}\alpha _{l}z$ for some $x,z\in D^{\ast }$ (similar, if $\alpha _{l}\pi _{k}\mid _{\tau }y$). There are three cases.\newline
Case 1: $x$ and $z$ are units. (simple)\newline
Case 2: $x$ and $z$, one is unit and the other is nonunit. (same steps Case 3 )\newline
Case 3: $x$ and $z$ are nonunits. Since $x,z\in D^{\#}$ and $D$ is $\tau $%
-n-atomic,
\begin{equation*}
y=x_{1}x_{2}...x_{p}\pi _{k}\alpha _{l}z_{1}z_{2}...z_{w}=\pi _{1}\pi_{2}...\pi _{k-1}\pi _{k}\pi _{k+1}...\pi _{m},
\end{equation*}
where $x_{q}(q=1,...,p)$ and $z_{r}(r=1,...,w)$ are $\tau $-normal irreducible elements. Since $\pi _{k}$ is normal, we have
\begin{equation*}
y=\pi _{k}x_{1}^{\prime }x_{2}^{\prime }...x_{p}^{\prime }\alpha_{l}z_{1}z_{2}...z_{w}=\pi _{k}\pi _{1}^{\prime }\pi _{2}^{\prime }...\pi_{k-1}^{\prime }\pi _{k+1}...\pi _{m}.
\end{equation*}
Thus
\begin{equation}
\frac{y}{\pi _{k}}=x_{1}^{\prime }x_{2}^{\prime }...x_{p}^{\prime }\alpha
_{l}z_{1}z_{2}...z_{w}=\pi _{1}^{\prime }\pi _{2}^{\prime }...\pi_{k-1}^{\prime }\pi _{k+1}...\pi _{m}, \label{EQ2}
\end{equation}
\newline
where $x_{n}^{\prime }(n=1,...,p)$ and $\pi _{m}^{\prime }(m=1,...,k-1)$ are $\tau $-normal irreducible elements by Lemma \ref{Lemm t}. Therefore equation (\ref{EQ2}) gives two distinct $\tau $-factorizations of $\frac{y}{\pi _{k}}$ into $\tau $-normal irreducibles because $D$ satisfies Condition (**), contradicting the minimality of $m$. Therefore $A=\emptyset $ and $D$ is a $\tau $-n-UFD.
\end{proof}

\bigskip

In the undirected case, using the same procedure used in Theorem \ref{main t}, we get the following result:
\begin{theorem}
Let $D$ be a domain with a symmetric, associate preserving, right multiplicative, and left cancelitive relation $\tau $ on $D^{\#}$. If $D$ is a $\tau $-n-atomic domain.  Consider the following statements:
\begin{enumerate}
\item $D$ is a $\tau $-n-UFD;
\item $G_{\tau }(x)$ is complete for all $x\in D^{\#}$;
\item $G_{\tau }(x)$ is connected for all $x\in D^{\#}$.
\end{enumerate}
Then
 \begin{center}
  $1 \Rightarrow 2 \Rightarrow 3$.   
 \end{center}
Moreover, if $D$ satisfies Condition (**), then all statements are equivalent.
\end{theorem}

The following example demonstrates that if $D$ is an atomic domain with a relation $\tau$ satisfies the conditions of Theorem \ref{main t}, then the Condition (**) is not redundant.
\begin{example}
  Let $D=\Bbb{Z}[1,i,j,k]$ and $\tau =D^{\#}\times D^{\#}$. Since $a=(1+i)$, $r=(1+j)$, and $r'=(1+k)$ are $\tau $-normal irreducible elements in $D$ such that $ar=r'a$, whereas $r$ and $r'$ are not associated. 
\end{example}
\bibliography{data2}

\begin{thebibliography}{1}

\bibitem{tauirreducible}
D.~D. Anderson and A.~M. Frazier.
\newblock On a general theory of factorization in integral domains.
\newblock {\em The Rocky Mountain Journal of Mathematics}, pages 663--705, 2011.

\bibitem{comprees}
M.~Axtell, N.~Baeth, and J.~Stickles.
\newblock Irreducible divisor graphs and factorization properties of domains.
\newblock {\em Communications in Algebra}, 39(11):4148--4162, 2011.

\bibitem{simplicial}
N.~Baeth and J.~Hobson.
\newblock Irreducible divisor simplicial complexes.
\newblock {\em Involve, a Journal of Mathematics}, 6(4):447--460, 2013.

\bibitem{irreducible}
J.~Coykendall and J.~Maney.
\newblock Irreducible divisor graphs.
\newblock {\em Communications in Algebra}, 35(3):885--895, 2007.

\bibitem{generalirreducible}
C.~P. Mooney.
\newblock Generalized irreducible divisor graphs.
\newblock {\em Communications in Algebra}, 42(10):4366--4375, 2014.

\bibitem{equivalentdef}
D.~Mordeson, M.~K. Sen, and D.~S. Malik.
\newblock Fundamentals of abstract algebra.
\newblock {\em The McCGraw-HILL Companies, Inc. New York st. Louis, san Francisco, printed in Singapore}, 1997.

\bibitem{paperphd1}
A.~Naser, R.~E. Abdel-Khalek, R.~M. Salem, and A.~M. Hassanein.
\newblock Types of irreducible divisor graphs of noncommutative domains.
\newblock {\em submitted}.

\bibitem{natomic}
A.~Naser, M.~H. Fahmy, and A.~M. Hassanein.
\newblock Non-commutative unique factorization rings with zero-divisors.
\newblock {\em Journal of Algebra and Its Applications}, 20(02):2150022, 2021.

\end{thebibliography}
\bibliographystyle{plain}
\end{document}